\documentclass[11pt]{amsart}

\usepackage{hyperref,amsmath,amsthm,amsfonts,amscd,flafter,epsf,amssymb}
\usepackage{epsfig}
\usepackage{amsfonts}
\usepackage{amssymb}
\usepackage{amsmath}
\usepackage{amsthm}
\usepackage{latexsym}
\usepackage{amscd}
\usepackage{flafter}
\usepackage{epsf}

\newcommand{\HF}{\widehat{\mathrm{HF}}}
\newcommand{\Sig}{\Sigma}
\newcommand{\ra}{\rightarrow}

\newcommand\Dual{\mathcal D}
\newcommand\Duality\Dual

\newcommand\relspinc{s}

\newcommand\ModSphere{\ModFlow\left({\mathbb S}\longrightarrow
\Sym^{g-1}(\Sigma_{1})\times \Sym^2(\Sigma_{2})\right)}
\newcommand\ModSpheres\ModSphere

\newcommand\UnparModSp{\widehat \ModSp}
\newcommand\UnparModFlow\UnparModSp

\newcommand\ModMaps{\mathcal M}
\newcommand\ModSp\ModMaps

\newcommand\spincrel\relspinc

\newtheorem{thm}{Theorem}[section]

\newtheorem{cor}[thm]{Corollary}

\newtheorem{lem}[thm]{Lemma}

\def\endproof{\relax\ifmmode\expandafter\endproofmath\else
  \unskip\nobreak\hfil\penalty50\hskip.75em\hbox{}\nobreak\hfil\bull
  {\parfillskip=0pt \finalhyphendemerits=0 \bigbreak}\fi}
\def\endproofmath$${\eqno\bull$$\bigbreak}
\def\bull{\vbox{\hrule\hbox{\vrule\kern3pt\vbox{\kern6pt}\kern3pt\vrule}\hrule}}

\newcommand{\Z}{\mathbb{Z}}

\newcommand{\ModSWfour}{\mathcal{M}}
\newcommand{\ModFlow}{\ModSWfour}

\newcommand{\SpinC}{{\mathrm{Spin}}^c}

\newcommand\abuts\Rightarrow
\newcommand\Sym{\mathrm{Sym}}

\begin{document}

\title[Seifert fibered 3-manifold with trivial HFH]{Seifert fibered homology spheres with trivial Heegaard Floer homology}%
\author{Eaman Eftekhary}%
\address{School of Mathematics, Institute for Research in Fundamental Sciences (IPM),
P. O. Box 19395-5746, Tehran, Iran}%
\email{eaman@ipm.ir}%

\thanks{}%
\keywords{Floer homology, Seifert fibered, homology spheres}%

\maketitle
\begin{abstract}
We show that among Seifert fibered integer homology spheres, Poincar\'e sphere (with either orientation) is the only non-trivial
example which has trivial Heegaard Floer homology. Together with an earlier result, this shows that if an integer homology sphere
has trivial Heegaard Floer homology, then it is a connected sum of a number of Poincar\'e spheres and hyperbolic homology spheres.
\end{abstract}
\section{Introduction}
Heegaard Floer homology, introduced by Ozsv\'ath and Szab\'o in \cite{OS-3mfld},
is an invariant of closed three-manifolds which has been relatively powerful in
distinguishing different manifolds from each other. Certain versions of these
invariants (hat version), which come in the form of graded abelian groups, may be computed
combinatorially by the result of Sarkar and Wang \cite{SW}. The advantage over the fundamental group is that it is immediate after
computation to decide whether the corresponding groups of two three-manifolds agree or
not, unlike non-abelian groups where identification of different presentations of the
group is a hard problem on its own.\\

It is interesting to ask a question similar to Poincar\'e conjecture about Heegaard Floer homology. {\emph{Is
there any non-trivial three-manifold with trivial Heegaard Floer homology (i.e. Heegaard Floer homology of $S^3$)?}}
Since Heegaard Floer homology package has a decomposition according to $\SpinC$ structures, this extra structure allows us
to distinguish non-homology spheres from the standard sphere. Yet, even among the integer homology spheres, one quickly
finds a counterexample: for the Poincar\'e sphere $P$, we have $\widehat{\mathrm{HF}}(P)= \widehat{\mathrm{HF}}(S^3)=\Z$.
The homological gradings on the two sides differ, but if we set $\tilde{P}=-P\#P$, there is an isomorphism of graded modules
$\HF(\tilde{P})=\HF(S^3)=(\Z)_0$. The same would be true for $P(n)=\#^n\tilde{P}$. But, are there any other examples?\\

 {\bf{Conjecture. }}{\emph{If for a homology sphere $Y$, the Heegaard Floer homology package (including
 the groups and their gradings) is isomorphic to the Heegaard Floer homology package of the sphere $S^3$,
 $Y$ is either the standard sphere or homeomorphic to one of $P(n)$ for some positive integer $n$.}}\\

This may sound quite unlikely; once we find one counter-example (i.e. $\tilde{P}$),
it is naively expected that many more counter-examples may be constructed. However, we
will provide strong evidence for the above conjecture. Namely, we will prove the following theorem, and 
a corollary of it which follows.

\begin{thm}
If $Y$ is a Seifert fibered integer homology sphere and $\HF(Y)=\Z$,
then $Y$ is either the standard sphere, or the Poincar\'e sphere with one of the two possible orientations.
\end{thm}
If $Y$ is a homology sphere with trivial Heegaard Floer homology, and $Y=Y_1\#...\#Y_n$ is the
prime decomposition of $Y$, then all $Y_i$ are prime homology spheres with $\HF(Y_i)=\Z$, by connected 
sum formula of Ozsv\'ath and szab\'o \cite{OS-properties}. Splicing formulas of the author in \cite{Ef-splicing} and the resulting 
surgery formulas of \cite{Ef-C-surgery}
may be used to show that no $Y_i$ can contain an incompressible torus (this is the main theorem of \cite{Ef-Essential}).
As a result, all $Y_i$ are geometric homology spheres, according to Thurston's geometrization conjecture (\cite{Thurston}), which is now 
Perelman's (and Morgan-Tian's) theorem (\cite{Per}, also see \cite{Morgan-Tian1} and \cite{Morgan-Tian2}).
Suppose that $Y_i$ is hyperbolic for $i=1,...,m$ and has one of the other $7$ geometries for $i>m$. Then the
above theorem implies that each $Y_i$ for $i>m$ is a Poincar\'e sphere, since all geometric homology spheres are
either hyperbolic or Seifert fibered. The above discussion may be re-stated as the following corollary.
\begin{cor}
If $Y$ is a prime homology sphere with $\HF(Y)=\Z$ and $Y$ is not one of $S^3, P$ or $-P$, it is a hyperbolic three-manifold.
\end{cor}
Thus, in order to prove the conjecture, we now need to prove the following reduced form.\\

{\bf{Conjecture (reduced form).}} {\emph{If $Y$ is a hyperbolic homology sphere, 
the rank of $\widehat{\mathrm{HF}}(Y)$ is bigger than $1$.}}\\

In the following section, we will recall the link between Seifert fibered homology spheres and plumbed three-manifolds. 
We will also quote the result of Ozsv\'ath and Szab\'o on computing Heegaard Floer homology of plumbed three-manifolds, 
from \cite{OS-plumbing}. Certain lemmas about plumbing diagrams of the standard sphere are proved in section $3$, which are
used in section $4$ to show that the number of generators for hat Heegaard Floer homology of $Y$ in the algorithm of Ozsv\'ath and Szab\'o 
is at least $2$, if the Seifert fibered homology sphere $Y$ is not the standard sphere or the Poincar\'e sphere.\\  

{\bf{Acknowledgement.}} The author would like to thank Yi Ni for some helpful discussions.\\

\section{Seifert fibered homology spheres and plumbing diagrams}
All oriented Seifert fibered homology spheres may be realized as the three-manifolds appearing in the boundary of 
the four-manifolds obtained by plumbing
disk bundles according to negative definite star-like graphs. Thus, Heegaard Floer homology of such manifolds may
be computed using the algorithm of  Ozsv\'ath and Szab\'o in \cite{OS-plumbing}. We recall the definition and construction
of \cite{OS-plumbing} in this section.\\

A graph $G$ equipped with an integer valued weight function $m$ on its
vertices is called a weighted graph. A weighted graph $G$ gives a
four-manifold $X(G)$ which is obtained by plumbing disk bundles over
spheres, corresponding to the vertices of $G$, according to the pattern
of the graph $G$. The disk bundle over the sphere corresponding to the vertex
$v$ is chosen so that its Euler number is $m(v)$. The sphere associated with $v$
is plumbed to the sphere associated with $w$ precisely when the two vertices
are connected in $G$ by an edge. The boundary of $X(G)$ is a closed three-manifold
which will be denoted by $Y(G)$. The homology group $\mathrm{H}_2(X(G),\Z)$ is
generated by vertices of $G$ and the intersection number $[v].[w]$ is $m(v)$
if $v=w$, is $1$ if $v$ and $w$ are connected in $G$ by an edge and $v\neq w$,  and is zero
otherwise. This gives an intersection matrix $M(G)$. A weighted graph is
negative definite if it is a disjoint union of trees and $M(G)$
is negative definite. A bad vertex is a vertex $v$ such that $m(v)>-d(v)$,
where $d(v)$ is the degree of $v$. Ozsv\'ath and Szab\'o give an algorithm
for computing $\mathrm{HF}^+(Y(G))$ from the combinatorics of the weighted graph $G$ 
when $G$ is a negative definite
graph with at most one bad vertex. In particular, they have a relatively easy
algorithm for computing the kernel of the shift map
$$U:\mathrm{HF}^+(Y(G);\Z/2\Z)\ra \mathrm{HF}^+(Y(G);\Z/2\Z).$$
Note that when  $\HF(Y)=\Z$ for a $3$-manifold $Y$, from the homology long exact sequence
(see \cite{OS-properties})
\begin{displaymath}
\dots \longrightarrow \HF(Y) \longrightarrow \mathrm{HF}^+(Y) \xrightarrow{U}
\mathrm{HF}^+(Y)\longrightarrow \HF(Y) \longrightarrow \dots
\end{displaymath}
and stabilization properties of $\mathrm{HF}^+(Y)$ we may conclude that 
$\mathrm{Ker}(U)=\Z/2\Z$.\\

Define an {\emph{association}} to be a map $n:V(G)\ra \Z$, where $V(G)$ is the set of vertices
of $G$, such that $n(v)$ has the same parity as $m(v)$ and $|n(v)|\leq -m(v)$. An association
is {\emph{initial}} if $m(v)< n(v) \leq -m(v)$ for all $v\in V(G)$, and is {\emph{final}}
if $m(v)\leq n(v)<-m(v)$ for every vertex $v\in V(G)$. We may {\emph{change}} an association $n$ to an association
$n'$ if there is some vertex $v$ such that the following three conditions are satisfied:\\
$1)$ $-n(v)=n'(v)=m(v)$,\\
$2)$ $n'(w)=n(w)+2$ for all $w\in V(G)$ which are connected to $v$ by an edge, \\
$3)$ $n(w)=n'(w)$ for all other vertices of $G$.\\
A \emph{good sequence} is a sequence $n_0,...,n_N$ of associations such that $n_0$ is initial and
$n_N$ is final, and such that $n_i$ may be changed to $n_{i+1}$ as above, for $0\leq i<N$.
Ozsv\'ath and Szab\'o show that the kernel of $U$ is generated by initial associations which
may be completed to good sequences, if
$G$ is a negative definite graph with at most one bad vertex. We will use this 
description of the kernel of the shift map on Heegaard Floer homology groups in the upcoming sections.\\

Seifert fibered homology spheres are given as $Y(G)$, where $G$ is a start-shape graph, with
a central vertex $v$ and tails $w^i_j$ of vertices, where $i=1,...,n,$ and  $j=1,...,p_i$,
such that $w^i_1$ is connected by an
edge to $v$ and each $w^i_j$ is connected to $w^i_{j-1}$ by an edge for $j>1$ (and there is no other
edge). Denote $m(w^i_j)$ by $m^i_j$ and $m(v)$ by $m$. Note that $m$ and all $m^i_j$ are negative. Since a
$-1$-sphere corresponding to a vertex with at most two neighbors in the graph may be blown down in the four-manifold $X(G)$
without changing the boundary $Y(G)$, 
we may assume that $m^i_j<-1$ for all $i,j$.
Thus, the graph would have at most one bad vertex, which is the vertex $v$. 
Define $\frac{a_i}{b_i}=[m^i_1;m^i_2;...;m^i_{p_i}]$ where
\begin{equation}
[m_1;m_2;...;m_{p}]
:=m_1-\cfrac{1}{m_2-\cfrac{1}{\ddots -\cfrac{1}{m_{p-1}-\cfrac{1}{m_{p}}}}}.
\end{equation}
In order for the resulting three-manifold $Y(G)$ to be a Seifert fibered manifold, 
this data should satisfy the following equation
\begin{equation}\label{eq:plumbing-condition}
a_1a_2...a_k(-m+\frac{b_1}{a_1}+\frac{b_2}{a_2}+...+\frac{b_k}{a_k})=1,
\end{equation}
and if this is the case, the corresponding three-manifold is  denoted by $\Sigma(a_1,...,a_k)$. The assumption that $Y(G)$ is a
homology sphere is equivalent to $\mathrm{det}(M(G))=\pm 1$. A weighted graph $G$ of the above form
will be called a negative definite star-like graph with $n$ \emph{rays}.
All homology Seifert fibered manifolds may be obtained in this way (see \cite{homology-spheres}).\\

If Heegaard Floer homology of $Y(G)$ is trivial(and consequently, the kernel of $U:\mathrm{HF}^+(Y(G))\ra \mathrm{HF}^+(Y(G))$
is $1$-dimensional), there is a unique initial association which may be completed
to a good sequence  associated with the
negative definite graph $G$. In other words, the initial element in all good sequences is the same.
If this is the case and $m<-1$, we will have $m(w)\leq -2$ for all vertices
$w\in V(G)$ and any association $n:V(G)\ra \Z$ with $m(w)<n(w)<-m(w)$ for all $w\in V(G)$ is a good sequence
of length one (i.e. $n$ is both initial and final). The interval $[m(w)+2,2-m(w)]$ consists of $-1-m(w)$ elements
with the same parity as $m(w)$. This implies that the rank of $\mathrm{Ker}(U)$ is at least equal to
$$d(G)=\prod_{w\in V(G)}(-1-m(w))\geq 1.$$
If  Heegaard Floer homology of $Y(G)$ is trivial, we should have $d(G)=1$, which implies that $m(w)=-2$ for all
$w\in V(G)$, i.e. all $m^i_j$ are equal to $-2$. With the above notation,
we will have $\frac{a_i}{b_i}=-\frac{p_i+1}{p_i}$.
Thus, equation~\ref{eq:plumbing-condition} reads as follows
\begin{equation}
\begin{split}
2-(\frac{p_1}{p_1+1}+\frac{p_2}{p_2+1}+&...+\frac{p_n}{p_n+1})\\
&=\frac{1}{(p_1+1)(p_2+1)...(p_n+1)}>0.
\end{split}
\end{equation}
Since $ \frac{1}{2}\leq \frac{p_i}{p_i+1}<1$, we should have $n=3$. It is then an easy exercise to check that
the only possible triple $(p_1,p_2,p_3)$ with $p_1\leq p_2\leq p_3$ satisfying the above equation is
$(1,2,4)$ which gives $P=\Sig(2,3,5)$ as $Y(G)$.\\

If $Y(G)$ is not the standard sphere or the Poincar\'e sphere, and Heegaard Floer homology of $Y(G)$ is
trivial, we may thus assume that $m=-1$. Since $-1$-spheres with at most two neighbors may be blown down
in a plumbing diagram, after simplification we may assume that $n\geq 3$. In order to show that for no such
negative definite graph $G$ the Heegaard Floer homology of $Y(G)$ is trivial, we will study the
combinatorics of certain plumbing diagrams for $S^3$ in the upcoming section.

\section{Plumbing diagrams of $S^3$}
In this section, we will consider star-like plumbing diagrams of $S^3$ of the  type
described in the previous section, with $2$ rays and $m=-1$.
These correspond to pairs of negative rational numbers  $(\frac{a_1}{b_1},\frac{a_2}{b_2})$, with $a_1,a_2>0$ such that
\begin{equation}\label{eq:sphere}
1+\frac{b_1}{a_1}+\frac{b_2}{a_2}=\frac{1}{a_1a_2}.
\end{equation}
This equation implies that precisely one of the two negative rational numbers $a_1/b_1$ and $a_2/b_2$ is
greater than or equal to $-2$ (or equivalently, that  one of the two negative rational
numbers $b_1/a_1$ and $b_2/a_2$ is less than or equal to $-1/2$ and the other one is greater than $-1/2$).
If $a_1/b_1=-2=[-2]$, it is easy to see that $a_2=1-2b_2$ and $a_2/b_2=[-3;-2,...;-2]$, where the number of $(-2)$s
in this continued fraction representation is equal to $-b_2-1$. 
Suppose that the quadruple $(a_1,b_1,a_2,b_2)$ satisfies equation~\ref{eq:sphere}, and $a_1/b_1> -2$.
It may be checked that under this assumption, the quadruple $(a_1',b_1',a_2',b_2')=(-b_1,2b_1+a_1,a_2+b_2,b_2)$
would be a new solution to equation~\ref{eq:sphere}.
Since we assume that $a_i$ are both positive and both $b_i$ are negative, both $a_i'$ would be positive and
both $b_i'$ would be negative. Since $b_1/a_1>-1$, we should have $0<a_1'<a_1$ 
and $0<a_2'=a_2+b_2<a_2$. This implies that the new quadruple is \emph{simpler}, in a sense, than the old quadruple.
In terms of the continued fractions, the above operation may be described as follows. If $a_1/b_1=[t_1;...;t_p]$ and
$a_2/b_2=[s_1;...;s_q]$, the assumption $a_1/b_1\geq -2$ implies that $t_1=-2$, and the assumption
$a_2/b_2<-2$ implies that $s_1<-2$. Then we would have $a_1'/b_1'=[t_2;t_3;...;t_p]$ and $a_2'/b_2'=[s_1+1;s_2;...;s_q]$.
 The new graph $G'$ corresponds
to a new four-manifold $X(G')$ which is obtained from $X(G)$ by blowing down a $-1$-sphere corresponding to the
vertex $v$. The boundary $Y(G')$ remains the same as $Y(G)=S^3$. We may thus check some of the claims
about such plumbing diagrams
by a simple induction. To do so, we should prove the claim for the quadruples $(2,-1,1+2k,-k)$, with $k\geq 1$, 
and prove that if the claim is true for $(a_1';b_1';a_2';b_2')$, it would be true for $(a_1;b_1;a_2;b_2)$.
The proof of the following lemma is an example of such inductions.\\
\begin{lem}\label{lem:1}
 If $a_1/b_1=[t_1;...;t_p]$ and $a_2/b_2=[s_1;...;s_q]$ satisfy equation~\ref{eq:sphere} then
\begin{equation}
\begin{split}
&\frac{1}{[t_1;...,t_{p-1};t_p+1]}+\frac{1}{[s_1;...;s_q]}\leq -1, \ \ \&\\
&\frac{1}{[t_1;...;t_p]}+\frac{1}{[s_1;...;s_{q-1};s_q+1]}\leq -1.
\end{split}
\end{equation}
\end{lem}
\begin{proof}
 We prove this lemma by induction. For the quadruple $(2,-1,2k+1,-k)$ we have $p=1$, $q=k$, $t_1=-2$, $s_1=-3$, and 
 $s_2=...=s_k=-2$, and both claims are easy to check. 
 Now suppose that the claim is true for $(a_1';b_1';a_2';b_2')$, and that $a_1/b_1=[t_1;...;t_p]$ and
$a_2/b_2=[s_1;...;s_q]$ are as above. Moreover, without loss of generality, assume that $t_1=-2$ and
that $a_1'/b_1'=[t_2;t_3;...;t_p]$ and $a_2'/b_2'=[s_1+1;s_2;...;s_q]$.
Let $u=[t_2;...;t_{p-1};t_p+1]$ and $w=[s_2;... ;s_q]$. Letting $s=s_1$ we should thus have (from the induction hypothesis)
\begin{equation}\label{eq:6}
\begin{split}
&\frac{1}{u}+\frac{1}{s+1-\frac{1}{w}}\leq -1\\
\Rightarrow\ &(s+2)uw+(s+1)w\geq 1+u.
\end{split}
\end{equation}
In order to prove the claim by induction, we should show
\begin{equation}\label{eq:7}
\begin{split}
&\frac{1}{-2-\frac{1}{u}}+\frac{1}{s-\frac{1}{w}}\leq -1\\
\Leftrightarrow\ &\frac{u}{2u+1}+\frac{w}{1-sw}\geq 1.
\end{split}
\end{equation}
Since $(s+2)uw+(s+1)w\geq 1+u$ by equation~\ref{eq:6}, the second inequality in equation~\ref{eq:7} 
is satisfied, and thus, so does the first inequality. The proof of the
second inequality in the statement of the lemma is identical. We just need to set $u=[t_2;...;t_p]$ and $w=[s_2;... ;s_{q-1},s_q+1]$.
\end{proof}
Suppose now that the  negative definite  two-ray star-like graph  $G$
 is given, and $a_1/b_1=[t_1;...;t_p]$ and $a_2/b_2=[s_1;...;s_q]$ are
the continued fractions corresponding to the two rays, as above, so that $Y(G)$ is the standard sphere.
Since the kernel of the map $U:\mathrm{HF}^+(S^3)\ra \mathrm{HF}^+(S^3)$
is generated by a unique element, every good sequence  associated with
the weighted graph $G$ should start with a unique initial association.
Denote one such good sequence by $n_0,n_1,...,n_N$, where $n_i:V(G)\ra \Z$ is an association for 
$i=0,1,...,N$.
So in particular, every other good association should start with $n_0$ too.
If $n:V(G)\ra \Z$ is an
initial association so that $n(w)\leq n_0(w)$ for all $w\in V(G)$,
it is easy to show that the sequence $n_0,...,n_N$ may be modified to
a good sequence $n=n_0',n_1',...,n_M'$ for some $M\leq N$. The uniqueness
of the initial association in the good sequences thus implies that there is no initial
association $n$ with $n\leq n_0$ and $n\neq n_0$, i.e.
$n_0(w)=2+m(w)$ for all $w\in V(G)$.  On the other hand, if $n_0,...,n_N$ is a good sequence,
so is $-n_N,-n_{N-1},...,-n_0$. Again, as a consequence of the uniqueness of the
initial association of good sequences, we  have $-n_N=n_0$. Thus
$n_N(w)=-m(w)-2$ for all $w\in V(G)$.\\
When we change $n_{i-1}$ to $n_{i}$ there
is a unique vertex $w_i\in V(G)$ with the property that $n_{i-1}(w_i)=-m(w_i)$ and $n_i(w_i)=m(w_i)$. We will call the sequence
$w_1,...,w_N$ the {\emph{sequence of vertices}} associated with the good sequence $n_0,...,n_N$.\\
\begin{lem}\label{lem:2}
If the graph $G$ is a negative definite star-like graph with two rays corresponding to the pair of negative rational numbers
$(a_1/b_1,a_2/b_2)$ with $a_1,a_2>0$ and with vertices $v,w^1_1,...,w^1_p,w^2_1,...,w^2_q$ as above,
the number of times the vertex $v$ appears in the sequence of vertices $w_1,...,w_N$ associated with any good sequence is
 equal to $a_1+a_2-1$.
\end{lem}
\begin{proof}
If $n,n':V(G)\ra \Z$ are two given functions, define $\langle n|n'\rangle=\sum_{w\in V(G)}n(w).n'(w)$. Let $A_i/B_i=[t_i;...;t_p]$
and $C_j/D_j=[s_j;...;s_q]$ for $i=1,...,p$ and $j=1,...,q$. Thus $A_1/B_1=a_1/b_1$ and $C_1/D_1=a_2/b_2$.
Assume that $A_i,C_i>0$ and $B_i,D_i<0$. Define $n(w^1_i):=C_1B_i$, $n(w^2_j):=A_1D_j$, and
$n(v)=-A_1C_1$. We may then compute $\langle n| n_k\rangle -\langle n| n_{k-1}\rangle$
for $k=1,...,N$ as follows:\\
$\bullet$ If $w_k=w^1_i$ for some $i=2,...,p$, we have
\begin{equation}
\begin{split}
\frac{1}{2}\big(\langle n| n_k\rangle -\langle n| n_{k-1}\rangle\big)&=\frac{1}{2}\sum_{w\in V(G)}n(w)(n_k(w)-n_{k-1}(w))\\
&=n(w^1_i)m(w^1_i)+n(w^1_{i-1})+n(w^1_{i+1})\\&=C_1(t_iB_i+B_{i-1}+B_{i+1}),
\end{split}
\end{equation}
where it is understood that $B_{p+1}=0$. From the definition of $A_\ell$ and $B_\ell$, it is clear that 
$\frac{A_\ell}{B_\ell}=t_\ell-\frac{B_{\ell+1}}{A_{\ell+1}}$ and thus
$B_{\ell}=-A_{\ell+1}$ and
$A_{\ell}=-t_\ell A_{\ell+1}+B_{\ell+1}$ for  all $\ell$. Combining these two relations, we obtain $t_iB_i+B_{i-1}+B_{i+1}=0$
for $i=2,...,p$ and $\langle n| n_k\rangle=\langle n| n_{k-1}\rangle$ in this case.\\
$\bullet$ If $w_k=w^1_1$, we would have
\begin{equation}
\begin{split}
\frac{1}{2}\big(\langle n| n_k\rangle -\langle n| n_{k-1}\rangle\big)&=
n(w^1_1)m(w^1_1)+n(v)+n(w^1_{2})\\&=C_1(t_1B_1-A_1+B_{2}),
\end{split}
\end{equation}
and $t_1B_1-A_1+B_2$ is zero with the same reasoning. Thus    $\langle n| n_k\rangle
=\langle n| n_{k-1}\rangle$.\\
$\bullet$ If $w_k=w^2_j$ for some $j=1,...,q$, we would have  $\langle n| n_k\rangle=\langle n| n_{k-1}\rangle$ with
a similar argument.\\
$\bullet$ If $w_k=v$, it is implied that
\begin{equation}
\begin{split}
\frac{1}{2}\big(\langle n| n_k\rangle -\langle n| n_{k-1}\rangle\big)&=
n(v)m(v)+n(w^1_1)+n(w^2_{1})\\&=A_1C_1+C_1B_1+A_1D_1\\&=a_1a_2+a_2b_1+a_1b_2=1,
\end{split}
\end{equation}
where the last equality is obtained from equation~\ref{eq:sphere}. \\
The above computations imply that whenever $v$ appears as $w_k$ in $w_1,...,w_N$, the value of $\langle n| n_k\rangle$
in comparison with $\langle n| n_{k-1}\rangle$ jumps by $2$.
Thus, the number of times $v$ appears in the
sequence of vertices $w_1,...,w_k$ is equal to $\frac{1}{2}\big(\langle n| n_N\rangle -\langle n| n_{0}\rangle\big)$.
Since  $n_0(w)=-n_N(w)=2+m(w)$ for all $w\in V(G)$, we may compute  
\begin{equation}\label{eq:main}
\begin{split}
\langle n| n_N\rangle -&\langle n| n_{0}\rangle
=\sum_{w\in V(G)}n(w)(n_N(w)-n_{0}(w))\\
&=2\Big[C_1\big(\sum_{i=1}^pB_i(-2-t_i)\big) +A_1\big(\sum_{j=1}^qD_j(-2-s_j)\big)+A_1C_1\Big]\\
&=2\Big[-C_1\big(\sum_{i=2}^p(B_{i-1}+t_iB_i+B_{i+1})\big)-C_1B_1(1+t_1)\\ &\ \ -C_1B_2-C_1B_p
 -A_1\big(\sum_{j=2}^q(D_{j-1}+s_jD_j+D_{j+1})\big)\\ &\ \ -A_1D_1(1+s_1)-A_1D_2-A_1D_q+A_1C_1\Big]\\
&=-2\Big[-A_1C_1+C_1B_1(1+t_1)+A_1D_1(1+s_1)\\ 
&\ \ \ \ \ \ \ \ \ \ \ \ \ \ \ \ \ \ \ \ \ \ \ \ \ \ \ \ \ \ \ +C_1B_2+A_1D_2-C_1-A_1 \Big].
\end{split}
\end{equation}
In the above computation, the last equality is the result of the equations $B_{i-1}+t_iB_i+B_{i+1}=0=D_{j-1}+s_jD_j+D_{j+1}$ for
$i=2,...,p$ and $j=2,..,q$ and the fact that $B_p=D_q=-1$.
Since $\frac{A_1}{B_1}=t_1-\frac{B_2}{A_2}$ and $\frac{C_1}{D_1}=s_1-\frac{D_2}{C_2}$
we have $B_2=A_1-t_1B_1$ and $D_2=C_1-s_1D_1$. Replacing these expressions for $B_2$ and $D_2$ in equation~\ref{eq:main}
we obtain
\begin{equation}
\begin{split}
\langle n| n_N\rangle -\langle n| n_{0}\rangle&=-2(A_1C_1+C_1B_1+A_1D_1-A_1-C_1)\\
&=2(A_1+C_1-1)=2(a_1+a_2-1).
\end{split}
\end{equation}
This completes the proof of lemma.
\end{proof}
\section{Proof of the main theorem}
We are now ready to prove the main theorem of this paper.
\begin{thm}
 If $Y$ is a Seifert fibered integer homology sphere and $\HF(Y)=\Z$,
then $Y$  is the Poincar\'e sphere with one of the two possible orientations.
\end{thm}
\begin{proof}
 We have already seen that if $Y$ is a Seifert fibered integer homology sphere different
from the Poincar\'e sphere and $\HF(Y)=\Z$, then $Y$ (with one of the two orientations) may be realized as $Y(G)$ for a
star-shape negative definite graph with vertices $v$ and $w^i_j$, $i=1,...,n$ and $j=1,...,p_i$, as
in the second section.
Moreover, we have seen that the weight function $m:V(G)\ra \Z^{<0}$ would be so that
$m(v)=-1$. As before, assume that  $m^i_j=m(w^i_j)$ and $q_i=a_i/b_i=[m^i_1;m^i_2;...;m^i_{p_i}]$ for
$i=1,...,n$, with $a_i>0$, and $q_1>q_2>...>q_n$. Note that this data should satisfy the
equation
\begin{equation}\label{eq:Seifert-numbers}
1+\frac{b_1}{a_1}+\frac{b_2}{a_2}+...+\frac{b_n}{a_n}=\frac{1}{a_1...a_n}.
\end{equation}
If all $m^i_1$ are less than $-2$, choosing $n(v)=1, n(w^i_1)\in[2+m^i_1,-m^i_1-4]$ and  $n(w^i_j)\in[2+m^i_j,-m^i_j-2]$ for
$j>1$, we get an initial association which may be completed to a good sequence $(n,n')$, where
$n'(v)=-1$, $n'(w^i_1)=n(w^i_1)+2\leq-m^i_1-2$, and $n'(w^i_j)=n(w^i_j)$ for $j>1$. Since all
good sequences start from a unique initial association, we should have $2+m^i_1=-m^i_1-4$,
and $2+m^i_j=-m^i_j-2$ for $j>1$. This means that
$m^i_1=-3$ and $m^i_j=-2$ for $i=1,...,n$ and $j>1$. We may thus compute $a_i=2p_i+1$ and $b_i=-p_i$ for
$i=1,...,n$. Equation~\ref{eq:Seifert-numbers} implies that
\begin{equation}
1-\big(\frac{p_1}{2p_1+1}+\frac{p_2}{2p_2+1}+...+\frac{p_n}{2p_n+1}\big)>0.
\end{equation}
Since $p_i/(2p_i+1)\geq 1/3$, the above inequality means that $n= 2$. But for $n=2$,
equation~\ref{eq:Seifert-numbers} implies that
$1-\frac{p_1}{2p_1+1}-\frac{p_2}{2p_2+1}=\frac{1}{(2p_1+1)(2p_2+1)}$
which is equivalent to $p_1+p_2=0$. Since this is a contradiction, at least one of $m^i_1,\ i=1,...,n$ 
should be equal to $-2$.
It is easy to see that $q_1>...>q_n$ implies $i=1$.\\
Let $p$ and $q$ be the largest numbers, so that for $A/B=[m^1_1,...,m^1_p]$ and $C/D=[m^2_1,...,m^2_q]$ we have
$AC+AD+BC=1$, i.e. the star-like negative definite graph with two rays corresponding to these two negative rational
numbers is the standard sphere. There are unique negative integers $t$ and $s$, so that the pairs of rational
numbers $([m^1_1,...,m^1_p, t],[m^2_1,...,m^2_q])$, $([m^1_1,...,m^1_p],[m^2_1,...,m^2_q,s])$, and
$([m^1_1,...,m^1_p, t],[m^2_1,...,m^2_q,s])$ all correspond to standard spheres.
If $m^1_{p+1}$ is greater than $t$, lemma~\ref{lem:1} implies that
\begin{equation}
 \frac{b_1}{a_1}+\frac{b_2}{a_2}\leq \frac{1}{[m^1_1;...;m^1_{p+1}]}+\frac{1}{[m^2_1;...;m^2_q]}\leq -1.
\end{equation}
This contradicts equation~\ref{eq:Seifert-numbers} and the fact that $b_i/a_i<0$, while $a_i>0$. Similarly
$m^2_{q+1}$ may not be greater than $s$. By the assumption on $p$ and $q$, the equality can not happen too.
Thus $m^1_{p+1}<t$ and $m^2_{q+1}<s$. Let $G'$ be the negative definite star-like graph with two rays
representing the standard sphere which corresponds to the pair of negative rational numbers
$([m^1_1;...;m^1_p;t],[m^2_1;...;m^2_q;s])$.\\
Let $A/B=[m^1_1;m^1_2;...;m^1_p]$ and $C/D=[m^2_1;m^2_2;...;m^2_q]$. It is then easy to check that
$a_1/b_1\geq A/B$ and $a_2/b_2\geq C/D$. Thus
$$1+\frac{b_1}{a_1}+\frac{b_2}{a_2}\leq  1+\frac{B}{A}+\frac{D}{C}=\frac{1}{AC}.$$
As a consequence, we should have $m^i_1\leq -AC-1$ for all $i=3,...,n$.\\
Since $Y(G)$ has trivial Heegaard Floer homology, we 
know that the initial association $n$ defined by $n(w)=2+m(w)$ for all $w\in V(G)$ may be completed
to a good sequence $n=n_0,n_1,...,n_N$, with the associated path of vertices $w_1,...,w_N$. Let $k$ be the
smallest integer with the property that $w_{k+1}\in\{w^i_j\ |\ i=3,...,n, j=1,..,p_i\}$. In particular we assume
$$w_1,...,w_k\in\{w^i_j\ |\ i=1,2, j=1,..,p_i\}\cup\{v\}.$$
This simply implies that
$w_{k+1}\in\{w^i_1\ |\ i=3,...,n\}$.
Our first claim is that
$$\{w_1,...,w_k\}\subset\mathcal{A}=\{w^1_1,...,w^1_p,w^2_1,...,w^2_q,v\}.$$
If this is not the case, let $w_\ell$ be the first element of this set which is not in $\mathcal{A}$.
Thus,  $w_\ell$ is one of the vertices in $\mathcal{B}=\{w^1_{p+1},w^2_{q+1}\}$.
The vertices in $G'$ are in correspondence with the vertices in $\mathcal{A}\cup\mathcal{B}$, and
we will abuse the notation by using the same names for its vertices. Construct
the sequence $n'_0,...,n'_{\ell-1}:V(G')\ra \Z$ by letting
\begin{equation}
n'_i(w)=\begin{cases}
n_i(w)\ \ \ \ \ \ \ \ \ \ \ \ \ \ \ \ \ \ \ \ \ \ \text{if }w\in \mathcal{A}\\
n_i(w^1_{p+1})-m^1_{p+1}+t\ \ \ \text{if }w=w^1_{p+1}\\
n_i(w^2_{q+1})-m^2_{q+1}+s\ \ \ \text{if }w=w^2_{q+1}
\end{cases}
\end{equation}
 Note that $n_0'$ is the initial association
in the unique good sequence for $S^3=Y(G')$, and all the changes in the sequence $n'_0,...,n'_{\ell-1}$
are allowed changes. Since $n_0'$ is the initial association in a good sequence, $n'_{\ell-1}$ would satisfy
$|n'_{\ell-1}(w)|\leq -m'(w)$ for all $w\in V(G')$, where $m'$ is the weight function for $G'$.
However, since $w_\ell\in \mathcal{B}$, we should have $n_{\ell-1}(w)=-m(w)$ for some
$w\in \mathcal{B}$, say $w^1_{p+1}$. But this implies that
$$n'_{\ell-1}(w^1_{p+1})=-m^1_{p+1}-m^1_{p+1}+t>-t=|m'(w^1_{p+1})|,$$
which is a contradiction, proving our first claim.\\
Let $G''$ be the negative definite star-like graph with two rays corresponding to the
pair of negative rational numbers $([m^1_1;...;m^1_p],[m^2_1;...;m^2_q])$, identify the
vertices of $G''$ with the elements of $\mathcal{A}$, and denote the weight function on the vertices of
$G''$ by $m'':V(G'')\ra \Z$. Then we may restrict $n_0,...,n_k$ to the vertices of $G''$ and
think of them as associations for this negative definite graph, and of $w_1,...,w_k$ as part of the
sequence of vertices associated with the good sequence extending $n_0,...,n_k$
(since they are all elements of $\mathcal{A}=V(G'')$).
The number of times $v$ appears is this sequence is less than or equal to the number of times
 $v$ appears in the sequence of vertices associated with  a good sequence associated with $G''$.
By lemma~\ref{lem:2} this later number is equal to
$A+C-1$. Thus, $n_{k}(w^i_1)\leq m^i_1+2A+2C$ for all $i=3,...,n$.
If $w_{k+1}=w^i_1$ for some $i\geq 3$, we have $n_k(w^i_1)=-m^i_1$ and thus $m^i_1\geq -A-C$.
But we already know that $m^i_1\leq -AC-1$. Combining these two inequalities, we have $A+C\geq AC+1$,
which is a contradiction since $A,C\geq 2$. The only remaining possibility is that
$$\{w_1,...,w_N\}\subset \mathcal{A}=\{w^1_1,...,w^1_p,w^2_1,...,w^2_q,v\}.$$
If this is the case, by restricting the associations $n_0,...,n_N$ to the vertices
of $G''$ we obtain a good sequence $n_0'',...,n_N'':V(G'')\ra \Z$, which would be a
good sequence  associated with $S^3=Y(G'')$ extending a unique initial association.
The number of times $v$ appears in the sequence of vertices $w_1,...,w_N$ is thus $A+C-1$. \\
Let $n_i':V(G)\ra \Z,\ i=0,...,N$ be new associations defined by $n_i'(w^3_1)=n_i(w^3_1)+2$ and
$n_i'(w)=n_i(w)$ for $w\neq w^3_1$. Note that $n_i'(w^3_1)\leq m^3_1+2(A+C)$, since $v$ appears
at most $A+C-1$ times in the sequence $w_1,...,w_N$. If $n_i'(w^3_1)\geq -m^3_1$, we would have
$-m^3_1\leq A+C$. This is a contradiction with what we proved earlier that $-m^3_1\geq AC+1$.
Thus, all $n_i'$ are in fact associations, and $n_N'$ is a final association. This means that
the sequence $n_0',...,n_N'$ is a good sequence starting from an initial association different from
$n_0$. So the rank of the kernel of $U:\mathrm{HF}^+(Y(G))\ra \mathrm{HF}^+(Y(G))$ is at least $2$.
This contradiction completes the proof of our main theorem.
\end{proof}


\end{document}